\documentclass[12pt]{amsart}
\usepackage{amscd,amsmath,amsthm,amssymb,enumerate}
\usepackage[left]{lineno}
\usepackage{pstricks}
\usepackage[utf8]{inputenc}
\usepackage{epstopdf}

\usepackage{comment}

\usepackage[all]{xy}

 \def\Soc{{\mathbf Soc}}

 \def\opn#1#2{\def#1{\operatorname{#2}}} 
 \opn\chara{char} \opn\length{\ell} \opn\pd{pd} \opn\rk{rk}
 \opn\projdim{proj\,dim} \opn\injdim{inj\,dim} \opn\rank{rank}
 \opn\depth{depth} \opn\grade{grade} \opn\height{height}
 \opn\bigheight{bigheight}
 \opn\embdim{emb\,dim} \opn\codim{codim}
 
 \opn\superheight{superheight}\opn\lcm{lcm}
 \opn\trdeg{tr\,deg}
 \opn\reg{reg} \opn\lreg{lreg} \opn\ini{in} \opn\lpd{lpd}
 \opn\size{size} \opn\sdepth{sdepth}
 \opn\link{link}\opn\fdepth{fdepth}\opn\lex{lex}
 \opn\type{type}
 \opn\gap{gap}
 \opn\arithdeg{arith-deg}
 \opn\Deg{Deg}
 \opn\sat{sat}
 \opn\mat{mat}
 \opn\Mat{Mat}
 \opn\div{div} \opn\Div{Div} \opn\cl{cl} \opn\Cl{Cl}
 \opn\Spec{Spec} \opn\Supp{Supp} \opn\supp{supp} \opn\Sing{Sing}
 \opn\Ass{Ass} \opn\Min{Min}\opn\Mon{Mon} \opn\Max{Max}
 \opn\Ann{Ann} \opn\Rad{Rad} \opn\Soc{Soc}
 \opn\Im{Im} \opn\Ker{Ker} \opn\Coker{Coker} \opn\Am{Am}
 \opn\Hom{Hom} \opn\Tor{Tor} \opn\Ext{Ext} \opn\End{End}
 \opn\Aut{Aut} \opn\id{id}
 
 \opn\nat{nat}
 \opn\pff{pf}
 \opn\Pf{Pf} \opn\GL{GL} \opn\SL{SL} \opn\mod{mod} \opn\ord{ord}
 \opn\Gin{Gin} \opn\Hilb{Hilb}\opn\sort{sort}
 \opn\PF{PF}\opn\Ap{Ap}
 \opn\mult{mult}
 \opn\bight{bight}
 \opn\aff{aff}
 \opn\relint{relint} \opn\st{st}
 \opn\lk{lk} \opn\cn{cn} \opn\core{core} \opn\vol{vol}  \opn\inp{inp} \opn\nilpot{nilpot}
 \opn\link{link} \opn\star{star}\opn\lex{lex}\opn\set{set}
 \opn\width{wd}
 \opn\Fr{F}
 \opn\QF{QF}
 \opn\G{G}
 \opn\type{type}\opn\res{res}
 \opn\conv{conv}
 \opn\Shad{Shad}
 \opn\gr{gr}
 
 \def\pot#1#2{#1[\kern-0.28ex[#2]\kern-0.28ex]}

 \opn\dirlim{\underrightarrow{\lim}}
 \opn\inivlim{\underleftarrow{\lim}}
 \let\to=\rightarrow
 
 \def\Implies{\ifmmode\Longrightarrow \else
         \unskip${}\Longrightarrow{}$\ignorespaces\fi}
 \def\implies{\ifmmode\Rightarrow \else
         \unskip${}\Rightarrow{}$\ignorespaces\fi}
 \def\iff{\ifmmode\Longleftrightarrow \else
         \unskip${}\Longleftrightarrow{}$\ignorespaces\fi}

 \let\:=\colon
\theoremstyle{plain}
\newtheorem{theorem}{Theorem}[section]
\newtheorem{Theorem}[theorem]{Theorem}
\newtheorem{thm}[theorem]{Theorem}

\newtheorem{prop}[theorem]{Proposition}

\newtheorem{cor}[theorem]{Corollary}

\newtheorem{lem}[theorem]{Lemma}
\newtheorem{claim}{Claim}

\theoremstyle{definition}

\newtheorem{defn}[theorem]{Definition}

\newtheorem{ex}[theorem]{Example}
\newtheorem{question}[theorem]{Question}
\newtheorem{Question}[theorem]{Question}

\newtheorem{rem}[theorem]{Remark}

\newtheorem{fact}[theorem]{Fact}
\newtheorem*{acknowledgments}{Acknowledgments}
\newtheorem*{observation}{Observation}

 \let\epsilon\varepsilon
 \let\kappa=\varkappa
 \def\qed{\ifhmode\textqed\fi
       \ifmmode\ifinner\quad\qedsymbol\else\dispqed\fi\fi}
 \def\textqed{\unskip\nobreak\penalty50
        \hskip2em\hbox{}\nobreak\hfil\qedsymbol
        \parfillskip=0pt \finalhyphendemerits=0}
 \def\dispqed{\rlap{\qquad\qedsymbol}}
 
 \opn\dis{dis}
 \def\pnt{{\raise0.5mm\hbox{\large\bf.}}}
 
 \opn\Lex{Lex}

\newcommand{\rmr}{\mathrm{r}}

\newcommand{\rmF}{\mathrm{F}}

\newcommand{\rmQ}{\mathrm{Q}}

\newcommand{\fkm}{\mathfrak{m}}

\def\ol{\overline}

\def\tr{\mathrm{tr}}

\title{Upper bound on the colength of the trace of the canonical module in dimension one}
\author{J\"{u}rgen Herzog}
\address{J\"urgen Herzog: Fachbereich Mathematik, Universit\"at Duisburg-Essen, Fakult\"at f\"ur Mathematik, 45117 Essen, Germany}
\email{juergen.herzog@uni-essen.de}

\author{Shinya Kumashiro}
\address{Shinya Kumashiro: National Institute of Technology, Oyama College
771 Nakakuki, Oyama, Tochigi, 323-0806, Japan}
\email{skumashiro@oyama-ct.ac.jp}

\thanks{The second author was supported by JSPS KAKENHI Grant Number JP21K13766.}

\begin{document}

\begin{abstract}
We study the upper bound of the colength of trace of the canonical module in one-dimensional Cohen-Macaulay rings. We answer the two questions posed by Herzog-Hibi-Stamate and Kobayashi.
\end{abstract}

\maketitle



\section{Introduction}\label{section1}

Let $H$ be an additive subsemigroup of $\mathbb{N}=\{0, 1,2, \dots\}$ with $0\in H$ such that $\mathbb{N}\setminus H$ is finite. 
Then, $H$ defines a $K$-algebra 
\[
R=K[H]=K[t^h : h\in H] \subseteq K[t], 
\]
where $K[t]$ is the polynomial ring over a field $K$. $R$ is called a {\it numerical semigroup ring} of $H$ over $K$. $R$ is a one-dimensional Cohen-Macaulay graded domain possessing a graded canonical module $\omega_R$. One of the most famous results in numerical semigroup rings is Kunz's characterization of the Gorenstein property (\cite{Kunz}). Starting from this result, there have been several attempts to measure how much a numerical semigroup deviates from being Gorenstein. In this article, we focus on the trace of the canonical module and investigate the upper bound on its colength.

Let 
\[
\tr_R(\omega_R)=\sum_{f\in \Hom_R(\omega_R, R)} \Im f
\]
denote the {\it trace} of the canonical module $\omega_R$ (see \cite{Lin, HHS}). It is known that the ideal $\tr_R(\omega_R)$ defines the non-Gorenstein locus of $R$ (\cite[Lemma 2.1]{HHS}). Hence, if $R$ is a numerical semigroup ring, then the length of $R/\tr_R(\omega_R)$ is finite and $R/\tr_R(\omega_R)=0$ if and only if $R$ is Gorenstein. From this perspective, the following two conjectures for the upper bound of the length of $R/\tr_R(\omega_R)$ have been proposed. Let $\ell_R(*)$ denotes the length.

\begin{Question} (\cite[Question 1.3]{HHS2} and \cite[Question 2.10]{Ko}) \label{mainquest}
\begin{enumerate}[{\rm (a)}] 
\item Does the inequality 
\[
\ell_R(R/\tr_R(\omega_R)) \le g(H) - n(H)
\] 
hold? Here, 
$g(H)=|\mathbb{N}\setminus H|$ and $n(H)=|\{h\in H : h<\rmF(H)\}|$ denote the number of {\it gaps} and the number of  {\it non-gaps}, respectively, and $\rmF(H)$ denotes the Frobenius number of $H$, that is, the largest integer in $\mathbb{N}\setminus H$. 
\item Does the inequality 
\[
\ell_R(R/\tr_R(\omega_R)) \le bg(R)
\] 
hold? Here, 
\[
bg(R)=\inf \left\{\ell_S(R/S) : \begin{matrix} \text{$S$ is a (one-dimensional) Gorenstein ring such that} \\
\text{$S\subseteq R$ is a birational extension}
\end{matrix}
\right\}
\]
denotes the {\it birational Gorenstein colength}.
\end{enumerate} 
\end{Question}

It is known that Question \ref{mainquest}(a) has a positive answer if either $R$ is Gorenstein (\cite[Lemma 1(f)]{FGH}) or the embedding dimension of $R$ is at most $3$ (\cite[Proposition 2.2]{HHS2}). For Question \ref{mainquest}(b), it is known that Question \ref{mainquest}(b) has a positive answer if $bg(R) \le 1$ (\cite[Proposition 3.6]{Ko}).
Among these progresses, our results answer Question \ref{mainquest}(a) and (b) as follows:

\begin{thm} 
\begin{enumerate}[{\rm (a)}] 
\item {\rm (Theorem \ref{thm} and Example \ref{counterex}):} 
\begin{enumerate}[{\rm (i)}] 
\item Question \ref{mainquest}(a) has a positive answer if $R$ has the Cohen-Macaulay type at most $3$. 
\item There exists a numerical semigroup ring with Cohen-Macaulay type $5$ such that Question \ref{mainquest}(a) has a negative answer.
\end{enumerate} 
\item {\rm (Corollary \ref{corthm3.2} and Example \ref{counterex3}):} 
\begin{enumerate}[{\rm (i)}] 
\item If $R$ is not Gorenstein, then $\ell_R(R/\tr_R(\omega_R)) \le 2bg(R) - 1$. 
\item For all $\ell\ge 2$, there exists a numerical semigroup ring $R$ with $bg(R)=\ell$ such that 
\[
\ell_R(R/\tr_R(\omega_R)) = 2bg(R) - 1.
\] 
\end{enumerate} 
Hence, Question \ref{mainquest}(b) has a positive answer if $bg(R) \le 1$, but not if $bg(R) \ge 2$.
\end{enumerate} 
\end{thm}

In Section \ref{section2} and Section \ref{section3} we explore Question \ref{mainquest}(a) and Question \ref{mainquest}(b), respectively. In Section \ref{section4} we explore the invariants $\ell_R(R/\tr_R(\omega_R)) $, $bg(R)$, and $g(H)-n(H)$ in a special class of numerical semigroup rings, say, far-flung Gorenstein rings. 

Throughout this article, $R$ denotes a one-dimensional Cohen-Macaulay ring possessing a canonical module $\omega_R$. $\ol{R}$ (resp. $\rmQ(R)$) denotes the integral closure of $R$ (resp. the total ring of fraction). A finitely generated $R$-submodule $I$ of $\rmQ(R)$ is called a {\it fractional ideal} of $R$. For fractional ideals $I$ and $J$, there is a canonical isomorphism $\Hom_R(I, J) \cong J:I$, where the colon is considered in $\rmQ(R)$ (see \cite[Lemma 2.1]{HK}). $\ell_R(*)$ stands for the length and $\rmr(R)$ stands for the Cohen-Macaulay type of $R$.

\begin{acknowledgments}
The authors are grateful to Dumitru I. Stamate for giving useful comments to improve this paper. 
\end{acknowledgments}

\section{Bound of Question \ref{mainquest}(a)}\label{section2}
Throughout this section, unless otherwise stated, let $H$ be an additive subsemigroup of $\mathbb{N}=\{0, 1,2, \dots\}$ with $0\in H$ such that $\mathbb{N}\setminus H$ is finite.  Let $R=K[H]$ be a numerical semigroup ring of $H$. $\rmF(H)$ denotes the Frobenius number of $H$, that is, the largest integer in $\mathbb{N}\setminus H$. Let 
\begin{align*}
g(H)=|\mathbb{N}\setminus H| \quad \quad  \text{and} \quad \quad  
n(H)=|\{h\in H : h<\rmF(H)\}|
\end{align*}
be the number of {\it gaps} and the number of {\it non-gaps}, respectively. It is known that 
\[
\omega_R=\sum_{\alpha\in \mathbb{N}\setminus H} Rt^{-\alpha}
\]
is a graded canonical module of $R$ (\cite[Example (2.1.9)]{GW}). We call
\begin{align*}
\mathrm{PF}(H)=&\{ \alpha \in \mathbb{N}\setminus H : \text{$\alpha+h\in H$ for all  $h\in H$}\}\\
=&\{ \alpha \in \mathbb{N}\setminus H : \fkm t^\alpha\subseteq R\}
\end{align*}
the set of {\it pseudo-Frobenius numbers of $H$}, where $\fkm$ stands for the graded maximal ideal of $R$. It is also known that $\omega_R$ is minimally generated by pseudo-Frobenius numbers of $H$, that is, $r=|\mathrm{PF}(H)|$ is the Cohen-Macaulay type of $R$. Set 
\[
\mathrm{PF}(H)=\{\alpha_1< \alpha_2< \dots< \alpha_r=\rmF(H)\}.
\]  

The following observation provides a way to compute the trace of the canonical module.

\begin{observation} (see for example \cite[before Lemma 2.2]{HKS2}) 
Set 
\[
C=t^{F(H)}\omega_R=\sum_{\alpha\in \mathrm{PF}(H)} Rt^{\rmF(H)-\alpha}. 
\]
Then, we obtain that $R\subseteq C\subseteq \ol{R}=K[t]$, where $\ol{R}$ denotes the integral closure of $R$, since $\alpha_1< \alpha_2< \dots< \alpha_r=\rmF(H)$. We can describe the trace $\tr_R(\omega_R)$ of the canonical module by $(R:C)C$, where the colon is considered in the total ring $\rmQ(R)$ of fraction, since the map 
\[
\Hom_R (\omega_R, R)\otimes_R \omega_R \to R; f\otimes x \mapsto f(x) \quad \text{for $f\in \Hom_R (\omega_R, R)$ and $x\in \omega_R$}
\] 
is identified by $(R:C) \otimes_R C\to R$, where $f\otimes x\mapsto fx$ for $f\in C:R$ and $x\in C$ (see \cite{HK}).
\end{observation}

In what follows, set $C=t^{F(H)}\omega_R$ as above. The following is known, but we include the proof for the convenience of the readers. 

\begin{lem} {\rm (cf. \cite{HK})} \label{Cdual}
Let $I$ and $J$ be fractional ideals of $R$, that is, finitely generated $R$-submodules of the total ring $\rmQ(R)$ of fraction. If $I\subseteq J$, then 
\[
\ell_R(J/I)=\ell_R((C:I)/(C:J)).
\]
\end{lem}

\begin{proof} 
By considering the $C$-dual $\Hom_R(-, C) \cong C:-$ of the exact sequence $0 \to I \xrightarrow{\iota} J \to J/I \to 0$, we obtain that
\[
0\to C:J \xrightarrow{\iota} C:I \to \Ext_R^1 (J/I, C) \to 0
\]
since $C$ is a canonical module. Hence, $(C:I)/(C:J) \cong \Ext_R^1 (J/I, C)$. It follows that $\ell_R((C:I)/(C:J))=\ell_R (\Ext_R^1 (J/I, C))=\ell_R(J/I)$ by the local duality theorem (see for example \cite[Theorem 3.5.8]{BH}). 
\end{proof}

\begin{lem}\label{lemlem}
$g(H)-n(H)=\ell_R(C/R)$.
\end{lem}

\begin{proof}
Note that $g(H)=\ell_R (\ol{R}/R)$ and $n(H)=\ell_R(R/(R:\ol{R}))$ by definitions. By Lemma \ref{Cdual}, $\ell_R (\ol{R}/R)=\ell_R ((C:R)/ (C:\ol{R}))$. Furthermore, we obtain that 
\[
C:R=C \quad \text{and} \quad R:\ol{R} =(C:C):\ol{R}=C:C\ol{R}=C:\ol{R},
\]
see \cite[page 19]{HK}. Hence, $\ell_R (\ol{R}/R)=\ell_R(C/(R:\ol{R}))$. 
Thus,
\[
g(H)-n(H)=\ell_R(C/(R:\ol{R})) - \ell_R(R/(R:\ol{R}))= \ell_R(C/R).
\]
\end{proof}

We use the following Gulliksen's result to prove Theorem \ref{thm}:

\begin{fact} (\cite[Theorem 1]{G}) \label{Gulliksen}
Let $A$ be an Artinian local ring and $M$ be a finitely generated faithful $A$-module. Suppose that the Cohen-Macaulay type $\rmr(A)$ of $A$ is at most $3$. Then, $\ell_A(M) \ge \ell_A(A)$ holds.
\end{fact}


\begin{Theorem}\label{thm}
If the Cohen-Macaulay type $r=\rmr(R)$ of $R$ is at most $3$, then we have 
\[
\ell_R(R/\tr_R(\omega_R)) \le g(H)-n(H).
\]
\end{Theorem}

\begin{proof}
Since $R:C\subseteq (R:C)C= \tr_R(\omega_R)$, we have $\ell_R(R/\tr_R(\omega_R)) \le \ell_R (R/(R:~C))$. 
If $r=2$, then it follows that $\ell_R(R/\tr_R(\omega_R)) \le \ell_R (R/(R:~C))=\ell_R(C/R)$ because $C/R$ is cyclic and the annihilator of $C/R$ is $R:C$. Hence, by Lemma \ref{lemlem}, we have $\ell_R(R/\tr_R(\omega_R)) \le g(H)-n(H)$. 

Suppose that $r=3$. By considering the $C$-dual $\Hom_R(-, C)=C:-$ of the exact sequence $0 \to R:C \xrightarrow{\iota} R \to R/(R:C) \to 0$,
we obtain the exact sequence
\[
0 \to C:R \xrightarrow{\iota} C:(R:C) \to \Ext_R^1(R/(R:C), C) \to 0.
\]
Note that $C:R=C:(C:C)=C$ and $C:(R:C)=C:((C:C):C)=C:~(C:~C^2)=C^2$. 
Hence, according to \cite[Theorem 3.3.7(b)]{BH}, we obtain that 
\[
\omega_{R/(R:C)} = \Ext_R^1(R/(R:C), C) \cong C^2/C.
\]
Because $r=3$, we can choose integers $0<a<b$ such that $C=\langle 1, t^a, t^b\rangle$. Then, 
\[
C^2=\langle 1, t^a, t^b, t^{2a}, t^{a+b}, t^{2b}\rangle.
\]
Therefore, $C^2/C$ is generated by $t^{2a}, t^{a+b}, t^{2b}$; hence, the Cohen-Macaulay type of $R/(R:C)$ is at most $3$. By noting that $C/R$ is a faithful $R/(R:C)$-module, Fact \ref{Gulliksen} proves that $\ell_R(R/(R:C)) \le \ell_R(C/R)$. It follows that $\ell_R(R/\tr_R(\omega_R))\le \ell_R(R/(R:~C)) \le \ell_R(C/R) = g(H)-n(H)$.
\end{proof}

In contrast to Theorem \ref{thm}, there exists a counterexample for Question \ref{mainquest}(a) when $r=5$:

\begin{ex}\label{counterex}
Let $H=\langle 13, 14, 15, 16, 17, 18, 21, 23\rangle$ and $R=K[H]$. Then, we can check that $\tr_R(\omega_R)=R:\ol{R}=t^{26}\ol{R}$ (\cite[Example 5.4(iii)]{HKS2}). Hence, $g(H)=17$, $n(H)=9$, and $\ell_R(R/\tr_R(\omega_R))=9$. Thus, we obtain that 
\[
\ell_R(R/\tr_R(\omega_R))=9 > 8=17-9=g(H)-n(H).
\]
\end{ex}

Example \ref{counterex} arises from far-flung Gorenstein rings. We later consider far-flung Gorenstein numerical semigroup rings, see Section \ref{section4}.

\begin{question} 
How about the case of $\rmr(R)=4$?
\end{question}

\section{Bound of Question \ref{mainquest}(b)}\label{section3}

In this section we consider the inequality of Question \ref{mainquest}(b). The inequality may originate from Ananthnarayan's result. For an Artininan local ring $A$, Ananthnarayan \cite[Definition 1.2]{A} defined the {\it Gorenstein colength} as 
\[
g(A)=\inf \{\ell_S(S)-\ell_S(A) : \text{$S$ is an Artinian Gorenstein local ring mapping onto $A$}\}.
\]
Then, he proved that $\ell_A(A/\tr_A(\omega_A)) \le g(A)$ (\cite[Corollary 3.8]{A}). 

As an analogue of the result for a one-dimensional Cohen-Macaulay local ring $R$, Kobayashi proposed Question \ref{mainquest}(b). Here, note that he did not assume that $R$ is a numerical semigroup ring. Thus, for a while, we only suppose that $(R, \fkm)$ is a one-dimensional Cohen-Macaulay local ring possessing the canonical module $\omega_R$. We call an extension $S\subseteq R$ of rings {\it birational} if $\rmQ(S)=\rmQ(R)$ and $R$ is finitely generated as an $S$-module.
Then, we call the non-negative integer
\[
bg(R)=\inf \left\{\ell_S(R/S) : \begin{matrix} \text{$S$ is a Gorenstein ring such that} \\
\text{$S\subseteq R$ is a birational extension}
\end{matrix}
\right\}
\]
the {\it birational Gorenstein colength} (see \cite[Definition 1.3]{Ko}). 

\begin{prop} \label{thm3.2}
Let $(R, \fkm)$ be a one-dimensional non-Gorenstein Cohen-Macaulay local ring having the canonical module $\omega_R$. Let $S$ be a Gorenstein ring such that $S\subseteq R$ is a birational extension. Then, 
\[
\ell_S(R/\tr_R(\omega_R)) = 2 \ell_S(R/S) - \ell_S(\tr_R(\omega_R)/(S:R)).
\]
\end{prop}

\begin{proof}
Note that $S$ is a local ring of $\dim S=\dim R =1$, because $R$ is finitely generated as an $S$-module. Furthermore, we obtain that $\omega_R \cong \Hom_R(R, S)$ by \cite[Theorem 3.3.7(b)]{BH}. $\Hom_S(R, S)\cong S:R \subseteq R$. 

\begin{claim}\label{claim1}
$\ell_S(S/(S:R))=\ell_S(R/S)$.
\end{claim}

\begin{proof}[Proof of Claim \ref{claim1}]
By applying the $S$-dual $\Hom_S(-, S)=S:-$ to the exact sequence 
$0 \to S \to R \to R/S \to 0$ of $S$-modules, we obtain that 
\[
0 \to S:R \to S \to \Ext_S^1(R/S, S) \to 0.
\]
Hence, by noting that $\omega_S=S$ since $S$ is Gorenstein, $\ell_S(S/(S:R)) =\ell_S(\Ext_S^1(R/S, S)) = \ell_S(R/S)$ by the local duality theorem (\cite[Theorem 3.5.8]{BH}). 
\end{proof}

Therefore, by looking at the following inclusions
\begin{small} 
\[
\xymatrix{
& R   &  \\
S \ar@{-}[ur] &  & \tr_R(\omega_R) \ar@{-}[ul]\\
& S:R, \ar@{-}[ur] \ar@{-}[ul] 
}
\]
\end{small}
we obtain that 
\begin{align*}
\ell_S(R/\tr_R(\omega_R)) =& \ell_S(R/S) + \ell_S(S/(S:R)) - \ell_S(\tr_R(\omega_R)/(S:R))\\
=&2 \ell_S(R/S) -\ell_S(\tr_R(\omega_R)/(S:R)).
\end{align*}
\end{proof}

\begin{cor} \label{corthm3.2}
Let $(R, \fkm)$ be a one-dimensional non-Gorenstein Cohen-Macaulay local ring having the canonical module $\omega_R$. Then, 
\[
\ell_R(R/\tr_R(\omega_R)) \le 2 bg(R) -1
\]
holds.
\end{cor}

\begin{proof} 
We may assume that $bg(R)$ is finite; hence, there exists a Gorenstein ring $S$ with $bg(R)=\ell_S(R/S)$ such that $S\subseteq R$ is a birational extension. 
Since the ring homomorphism $S/(\fkm\cap S) \to R/\fkm$ is injective, 
\[
\ell_R(R/\tr_R(\omega_R))\le \ell_R(R/\tr_R(\omega_R)){\cdot}\ell_S(R/\fkm)= \ell_S(R/\tr_R(\omega_R)).
\]
It follows that $\ell_R(R/\tr_R(\omega_R)) \le 2bg(R) -\ell_S(\tr_R(\omega_R)/(S:R))$ by Proposition \ref{thm3.2}. 
On the other hand, when $\tr_R(\omega_R)=S:R$, we have $\tr_R(\omega_R)\cong \omega_R$ since $S:R\cong \omega_R$. Hence, it is enough to prove that $\tr_R(\omega_R)\cong \omega_R$ implies that $R$ is Gorenstein. Suppose that $\tr_R(\omega_R)\cong \omega_R$. Then we obtain the isomorphisms 
\[
R\cong \Hom_R(\omega_R, \omega_R)\cong \Hom_R(\tr_R(\omega_R), \tr_R(\omega_R))\cong \tr_R(\omega_R): \tr_R(\omega_R)
\] 
(note that $\tr_R(\omega_R)$ contains a non-zerodivisor of $R$ since $S:R\subseteq \tr_R(\omega_R)$). It follows that $\tr_R(\omega_R): \tr_R(\omega_R) = \alpha R$ for some $\alpha \in \rmQ(R)$. Then, $\alpha  \in \alpha R=\tr_R(\omega_R):~\tr_R(\omega_R)$. On the other hand, since $1\in \tr_R(\omega_R):~\tr_R(\omega_R)= \alpha R$, $\alpha^{-1} \in R \subseteq \tr_R(\omega_R): \tr_R(\omega_R)$. Thus, $\alpha$ is a unit of the endomorphism algebra $\tr_R(\omega_R):~\tr_R(\omega_R)$. Therefore, 
\[
R=\alpha^{-1}(\tr_R(\omega_R): \tr_R(\omega_R)) = \tr_R(\omega_R): \tr_R(\omega_R) = R: \tr_R(\omega_R),
\]
where the last equality follows by \cite[Proposition 2.8(vi)]{Lin}.
By considering the $R$-dual $R:- = \Hom_R(-, R)$ of the exact sequence $0 \to \tr_R(\omega_R) \to R \to R/\tr_R(\omega_R) \to 0$, it follows that $\Ext_R^1(R/\tr_R(\omega_R), R)=0$. Thus, by using the Rees lemma and the fact that $\tr_R(\omega_R)$ is an $\fkm$-primary ideal of $R$, $\tr_R(\omega_R)=R$. Hence, $R$ is Gorenstein by \cite[Lemma 2.1]{HHS}. 

Therefore, by noting that $S:R\cong \omega_R$, we obtain that $\ell_S(\tr_R(\omega_R)/(S:R))>0$ if $R$ is not Gorenstein. Thus, 
\[
\ell_R(R/\tr_R(\omega_R)) \le 2bg(R) -\ell_S(\tr_R(\omega_R)/(S:R)) \le 2bg(R) -1
\] 
as desired.
\end{proof}

The following examples say that the inequality of Corollary \ref{corthm3.2} is sharp. In particular, Question \ref{mainquest}(b) has a negative answer:

\begin{ex}\label{counterex2}
Let $H=\langle 10, 11, 12, 13, 14, 17\rangle$ and $R=K[[H]]$. Then, we can check that $\tr_R(\omega_R)=(t^{10}, t^{11}, t^{13}, t^{14}, t^{29})$. Thus, we obtain that $\ell_R(R/\tr_R(\omega_R))=3$. On the other hand, we can also check that $H'=\langle 10, 11, 12, 13, 14\rangle$ is symmetric and $\ell_R(R/K[[H']])=2$; hence, $bg(R)\le 2$. This follows that 
\[
3=\ell_R(R/\tr_R(\omega_R)) \le 2 bg(R) -1 \le 3.
\] 
Therefore, $bg(R)=2 < 3 = \ell_R(R/\tr_R(\omega_R))$. Thus, Question \ref{mainquest}(b) has a negative answer.
\end{ex}

We can generalize Example \ref{counterex2} as follows:

\begin{ex}\label{counterex3}
Let $\ell\ge 0$, and let 
\[
H=\langle m \ \ : \ \ 6\ell+10 \le m \le 9\ell +14 \text{\ \  or \ \ $m=9\ell+17 +3s$ for $0 \le s \le 3\ell$}\rangle.
\]
Set $R=K[[H]]$. Then, $\ell_R(R/\tr_R(\omega_R))=2\ell + 3$. On the other hand, 
\[
H'=\langle m \ \ : \ \ 6\ell+10 \le m \le 9\ell +14 \rangle
\] 
is symmetric and $\ell_R(R/K[[H']])=\ell+2$; hence, $bg(R)\le \ell + 2$. This follows that 
\[
2\ell + 3=\ell_R(R/\tr_R(\omega_R)) \le 2 bg(R) -1 \le 2\ell + 3.
\] 
Therefore, $bg(R)=\ell + 2 < 2\ell + 3 = \ell_R(R/\tr_R(\omega_R))$ for all $\ell \ge 0$. Thus, the inequality of Corollary \ref{corthm3.2} is sharp.
\end{ex}

 

\section{the invariants in far-flung Gorenstein rings}\label{section4}
In this section we study the invariants $\ell_R(R/\tr_R(\omega_R))$, $bg(R)$, and $g(H)-n(H)$ in a special class of numerical semigroup rings. Since $bg(R)$ is defined only in local rings, we reuse the notation of Section \ref{section2} for the completion $K[[H]]$ of the numerical semigroup ring $K[H]$. Set $R=K[[H]]$. Write
\[
H=\{a_0=0<a_1<a_2<\cdots <a_{n}< \cdots \}.
\]
Note that $a_{n(H)+i}=a_{n(H)} + i$ for all $i\ge 0$. 

\begin{prop} \label{prop4.1}
$bg(R) \le n(H) = \ell_R(R/(R:\ol{R}))$.
\end{prop}

\begin{proof}
Set 
\begin{align*} 
H'=\langle a_0, a_{n(H)}, a_{n(H)+1}, a_{n(H)+2}, \dots, a_{2n(H)-2}\rangle.
\end{align*}
 Then 
 \begin{align*} 
 \mathbb{N}\setminus H'=&\{i : 1\le i \le a_{n(H)}-1\} \cup  \{a_{2n(H)-1}=a_{n(H)}+n(H)-1\}. 
 \end{align*}
 Hence, $\mathrm{PF}(H)=\{a_{2n(H)-1}\}$. It follows that $H'$ is symmetric, that is, $K[[H']]$ is Gorenstein. By noting that 
 \[
 H\setminus H' = \{a_1, a_2, \dots, a_{n(H)-1}, a_{2n(H)-1}\}, 
 \]
 we obtain that $\ell_{K[[H']]}(R/K[[H']])=n(H)$. Hence, $bg(R) \le n(H)$.
 \end{proof}

In general, it would be difficult to confirm that $bg(R)$ is no longer small. But, if $R$ is a far-flung Gorenstein ring, then we have $bg(R) = n(H)$ (see Proposition \ref{prop4.3}). Let us recall the notion of far-flung Gorenstein rings:

\begin{defn} (\cite[Definition 2.3]{HKS2})
We say that a numerical semigroup ring $R$ is a {\it far-flung Gorenstein ring} if $\tr_R(\omega_R) = R:\ol{R}$. 
\end{defn}

Note that for an arbitrary numerical semigroup ring $R$, we have $\ell_R(R/\tr_R(\omega_R)) \le n(H)$ (e.g. \cite[Proposition A.1]{HHS2}, \cite[Lemma 3.2]{Kuma}). Hence, to consider the upper bound of $\ell_R(R/\tr_R(\omega_R))$, far-flung Gorenstein rings were a good trial run (recall Example \ref{counterex}). On the other hand, Question \ref{mainquest}(b) has a positive answer for far-flung Gorenstein rings -- even though Question \ref{mainquest}(b) has a negative answer in general.

\begin{prop} \label{prop4.3}
If $R$ is a far-flung Gorenstein numerical semigroup ring, then $bg(R)=n(H)=\ell_R(R/\tr_R(\omega_R))$. 
\end{prop}

\begin{proof}
Let $S$ be a Gorenstein ring such that $S\subseteq R$ is a birational extension. 
 By Proposition \ref{thm3.2}, we have 
\begin{align} \label{aaaa}
\begin{split} 
\ell_R(R/\tr_R(\omega_R)) =& \ell_S(R/\tr_R(\omega_R)) = 2 \ell_S(R/S) - \ell_S(\tr_R(\omega_R)/(S:R))\\
=&2 \ell_R(R/S) - \ell_R(\tr_R(\omega_R)/(S:R)).
\end{split}
\end{align}
Since $R$ is far-flung Gorenstein, $\tr_R(\omega_R) = R:\ol{R} = t^{c}\ol{R}$, where $c=F(H) +1$ is the conductor of $H$. Since $S:R\subseteq \tr_R(\omega_R) = t^{c}\ol{R}$, we can write $S:R=t^c X$, where $X=t^{-c} (S:R)\subseteq \ol{R}$. Then, we obtain that 
\[
\ell_R(\tr_R(\omega_R)/(S:R)) = \ell_R(t^c\ol{R}/t^cX)= \ell_R(\ol{R}/X) = \ell_R((X:X)/(X:\ol{R})),
\]
where the third equality follows by the same argument which are used for the proof of Lemma \ref{Cdual}, because $X\cong C$. Furthermore, we have 
\[
X:X=R \quad \text{and} \quad R:\ol{R} = (X:X):\ol{R}= X:X\ol{R} \supseteq X:\ol{R},
\]
where the last inclusion follows from $X\ol{R}\subseteq \ol{R}{\cdot}\ol{R}=\ol{R}$ since $X\subseteq \ol{R}$.
Therefore, $\ell_R(\tr_R(\omega_R)/(S:R)) =\ell_R(R/(X:\ol{R})) \ge \ell_R(R/(R:\ol{R})) = n(H)$. By noting that $\ell_R(R/\tr_R(\omega_R))=\ell_R(R/(R:\ol{R})) =n(H)$, \eqref{aaaa} provides that 
\[
n(H) =\ell_R(R/\tr_R(\omega_R)) \le 2 \ell_R(R/S) - n(H).
\]
Thus, $n(H) \le \ell_R(R/S)$ for each Gorenstein ring $S$ such that $S\subseteq R$ is a birational extension. This follows that $n(H) \le bg(R)$. By combining with Proposition \ref{prop4.1}, we have the conclusion.
 \end{proof}

Due to Proposition \ref{prop4.3}, we obtain examples of rings with $bg(R) > g(H) - n(H)$ and rings with $bg(R) < g(H) - n(H)$ as follows.

\begin{ex} \label{exfinal}
\begin{enumerate}[{\rm (a)}] 
\item Let $H_1=\langle 13, 14, 15, 16, 17, 18, 21, 23\rangle$ and $R_1=K[[H_1]]$. Then, as we saw in Example \ref{counterex}, $R_1$ is a far-flung Gorenstein ring (\cite[Example 5.4 (iii)]{HKS2}) and $\ell_{R_1}(R_1/\tr_{R_1}(\omega_{R_1}))=9 > 8=17-9=g(H_1)-n(H_1)$. Thus, we obtain that $bg(R_1)>g(H_1)-n(H_1)$ by Proposition \ref{prop4.3}.
\item Let $H_2=\langle 5, 6, 13, 14\rangle$ and $R_2=K[[H_2]]$. Then, $R_2$ is a far-flung Gorenstein ring (\cite[Theorem 6.4(ii)(1-2)]{HKS2}), and $\ell_{R_2}(R_2/\tr_{R_2}(\omega_{R_2}))=3 <4=7-3=g(H_2)-n(H_2)$. Thus, we obtain that $bg(R_2) < g(H_2)-n(H_2)$ by Proposition \ref{prop4.3}.
\end{enumerate} 
\end{ex}

\begin{rem} 
For a given ring $R$, the Gorenstein subring $S$ of $R$ for which $bg(R)=\ell_R(R/S)$ are not uniquely determind. Indeed, in Example \ref{exfinal}(b), at least $3$ Gorenstein subrings
\begin{align*} 
S_1=&K[[t^5, t^6]], \quad S_2=K[[t^6, t^{10}, t^{11}, t^{14}, t^{15}]], \quad \text{and}  \quad \\
S_3=&K[[t^{10}, t^{11}, t^{12}, t^{13}, t^{14}, t^{15}, t^{16}, t^{17}, t^{18}]]
\end{align*}
provide $bg(R_2)=3$. 
\end{rem}


\end{document}